\documentclass[a4paper,reqno]{amsart}

\pdfoutput=1

\usepackage{amsthm}
\newtheorem{theorem}{Theorem}[section]
\newtheorem{lemma}[theorem]{Lemma}
\newtheorem{corollary}[theorem]{Corollary}

\theoremstyle{definition}

\newtheorem{remark}[theorem]{Remark}
\newtheorem{definition}[theorem]{Definition}
\usepackage{amsmath}
\usepackage{amssymb}
\usepackage[english]{babel}
\usepackage{enumitem}
\usepackage[T1]{fontenc}
\usepackage{hyperref}
\usepackage{listings}
\usepackage{mathtools}
\usepackage{multicol}
\usepackage{stmaryrd}
\usepackage{xspace}

\newcommand{\sigbu}{\Sigma^\bullet}
\newcommand{\pibu}{\Pi^\bullet}
\newcommand{\pibucurried}[2]{\pibu_{#1}{#2}}

\newcommand{\pair}[2]{\left( #1 \; , \; #2 \right)}
\newcommand{\bigpair}[2]{\big( #1 \; , \; #2 \big)}
\newcommand{\functioncompose}{\ensuremath{\mkern-1mu \circ \mkern-1mu}}
\newcommand{\PointedPred}[1]{\operatorname{Fam_{#1}^{\bullet}}}

\newcommand{\ie}{i.e.\ }
\newcommand{\cf}{see}

\newcommand{\prd}[1]{\Pi_{#1}}
\newcommand{\sm}[1]{\Sigma_{#1}}

\newcommand{\lam}[1]{\lambda #1 .}

\newcommand{\blank}{\mathord{\hspace{1pt}\text{--}\hspace{1pt}}}

\newcommand{\jdeq}{\equiv}
\newcommand{\defeq}{\vcentcolon\equiv}
\newcommand{\eqvsym}{\simeq}
\newcommand{\eqv}[2]{\ensuremath{#1 \eqvsym #2}\xspace}
\newcommand{\eqvauto}[1]{\eqv{#1}{#1}}

\newcommand{\iseqown}[1]{\ensuremath{\mathsf{e}_{#1}}\xspace}
\newcommand{\eqvsymspace}{\enspace \eqvsym \enspace}
\newcommand{\eqvspace}[2]{\ensuremath{#1 \eqvsymspace #2}\xspace}

\newcommand{\idfunc}[1][]{\ensuremath{\mathsf{id}_{#1}}\xspace}

\newcommand{\idsym}{{=}}
\newcommand{\id}[3][]{\ensuremath{#2 =_{#1} #3}\xspace}

\newcommand{\idtypevar}[1]{\ensuremath{\mathsf{Id}_{#1}}\xspace}
\newcommand{\idauto}[2][]{\id[#1]{#2}{#2}}

\newcommand{\refl}[1]{\ensuremath{\mathsf{refl}_{#1}}\xspace}

\newcommand{\ct}{%
  \mathchoice{\mathbin{\raisebox{0.5ex}{$\displaystyle\centerdot$}}}%
             {\mathbin{\raisebox{0.5ex}{$\centerdot$}}}%
             {\mathbin{\raisebox{0.25ex}{$\scriptstyle\,\centerdot\,$}}}%
             {\mathbin{\raisebox{0.1ex}{$\scriptscriptstyle\,\centerdot\,$}}}
}

\newcommand{\opp}[1]{\mathord{{#1}^{-1}}}
\newcommand{\trans}[2]{\ensuremath{{#1}_{*}\mathopen{}\left({#2}\right)\mathclose{}}\xspace}
\newcommand{\isequiv}{\ensuremath{\mathsf{isequiv}}}

\newcommand{\UU}{\ensuremath{\mathcal{U}}\xspace}
\newcommand{\UUp}{{\UU_\bullet}}
\newcommand{\UUtt}[2]{\ensuremath{ \UU_{#1}^{\text{\scalebox{0.5}{$\leq$}} #2}}}
\newcommand{\UUt}[1]{\ensuremath{ \UU^{\text{\scalebox{0.5}{$\leq$}} #1}}}

\newcommand{\N}{\ensuremath{\mathbb{N}}\xspace}
\newcommand{\emptyt}{\ensuremath{\mathbf{0}}\xspace}
\newcommand{\unit}{\ensuremath{\mathbf{1}}\xspace}
\newcommand{\ttt}{\ensuremath{\star}\xspace}
\newcommand{\bool}{\ensuremath{\mathbf{2}}\xspace}
\newcommand{\btrue}{{1_{\bool}}}
\newcommand{\bfalse}{{0_{\bool}}}
\newcommand{\swap}{\ensuremath{\textsf{swap}}}

\newcommand{\iscontr}{\ensuremath{\mathsf{isContr}}}
\newcommand{\isprop}{\ensuremath{\mathsf{isProp}}}
\newcommand{\isset}{\ensuremath{\mathsf{isSet}}}

\def\compare#1#2#3#4{\if#1#3\if#2#41\else0\fi\else0\fi}

\newcommand{\istype}[1]{
  \edef\a{\compare-2#1\empty\empty}
  \if\a1 \iscontr \else
  \edef\b{\compare-1#1\empty\empty}
  \if\b1 \isprop \else
  \edef\c{#1}
  \if0\c \isset \else
  \mathsf{is}\mbox{-}{#1}\mbox{-}\mathsf{type} \fi\fi\fi
}

\newcommand{\idtoeqv}{\ensuremath{\mathsf{idtoeqv}}\xspace}

\newcommand{\Sn}{\mathbb{S}}
\newcommand{\base}{\ensuremath{\mathsf{base}}\xspace}
\newcommand{\lloop}{\ensuremath{\mathsf{loop}}\xspace}

\newcommand{\happly}{\ensuremath{\mathsf{happly}}}
\newcommand{\brck}[1]{\left\lVert #1 \right\rVert}

\newcommand{\fst}{\pi_1 \mkern-1mu}
\newcommand{\snd}{\pi_2 \mkern-1mu}

\newcommand{\pted}[1]{\mathfrak #1}
\newcommand{\loops}{\mathsf{Loop}}

\newcommand{\boolpower}[1]{\bool^{(#1)}}

\title[Higher Homotopies in a Hierarchy of Univalent Universes]{Higher Homotopies in a Hierarchy \\ of Univalent Universes}
\author{Nicolai Kraus}
\address{University of Nottingham}
\email{ngk@cs.nott.ac.uk}
\author{Christian Sattler}
\address{University of Nottingham}
\curraddr{University of Leeds}
\email{c.sattler@leeds.ac.uk}

\begin{document}

\begin{abstract}
For Martin-L\"of type theory with a hierarchy $\UU_0 : \UU_1 : \UU_2 : \ldots$ of univalent universes, we show that $\UU_n$ is not an $n$-type. 
Our construction also solves the problem of finding a type that \emph{strictly} has some high truncation level without using higher inductive types. 
In particular, $\UU_n$ is such a type if we restrict it to $n$-types.

We have fully formalized and verified our results within the dependently typed language and proof assistant Agda.
\end{abstract}

\maketitle

\vspace*{-24pt}

\section{Introduction} \label{sec:intro}

One of the most basic and well-known implications of Voevodsky's univalence axiom~\cite{voevodsky_univalentFoundationsProjectNSF} is that the first type universe, written $\UU_0$, does not have unique identity proofs.
This is due to the fact that, for an example, the type $\bool$ of boolean values is equivalent to itself in two different ways.
These equivalences give rise to two different inhabitants of $\id[\UU_0] \bool \bool$, a type that is sometimes written as $\idtypevar {{\UU_0}} (\bool, \bool)$.
In the language of homotopy type theory (HoTT)~\cite{HoTTbook}, this means that $\UU_0$ is not a \emph{set} or a \emph{$0$-type}.

This statement holds in standard Martin-L\"of type theory (MLTT) with one univalent universe.
If we have a hierarchy $\UU_0 : \UU_1 : \UU_2 : \ldots$ of univalent universes, it is natural to ask what we can say about the truncation levels of higher universes.
Reading through the argument that $\UU_0$ is not a $0$-type and choosing a number $n$, it seems plausible that the hierarchy allows the construction of types that can be shown to be not $n$-truncated.

However, this turns out to be fairly involved.
The question of how it could be done was discussed at the Special year on Univalent Foundations at the Institute for Advanced Study in Princeton (2012--2013).
As sketched above, the type $\bool$ is sufficient to see that $\UU_0$ is not a $0$-type.
To go further, one idea that was suggested by several people (first by Finster and Lumsdaine, as far as we know) was to consider the type of types that are \emph{merely equal} to $\bool$, written 
$\sm {X:\UU_0} \brck{\id[\UU_0] X \bool}$, 
where $\brck \blank$
is the \emph{propositional truncation}.
Technically, propositional truncation is an additional concept which is not available in the considered setting, but it can be encoded in a suitable way so that the construction can be carried out in plain MLTT with univalence.
The idea of $\sm {X:\UU_0} \brck{\id[\UU_0] X \bool}$ is to take $\bool$, but ``wrap'' it once, defining some sort of ``subtype'' (or ``subuniverse'') of $\UU_0$ that only contains $\bool$.
This operation shifts the non-trivial proof of $\id[\UU_0] \bool \bool$ by one level.
Finster and Lumsdaine used the construction to show that $\UU_1$ is not a $1$-type.
They also tried to repeat the ``wrapping'' in order to get the corresponding statements for higher universes, but this became difficult very quickly, and it was unclear whether the strategy could be used to prove a general statement.
Another argument for the fact that $\UU_1$ is not a $1$-type was given by Coquand, using the type of $\mathbb Z / 2\mathbb Z$-sets, \ie sets together with involutions.
It is not clear how a generalization of this construction could be used for higher cases. 

In this article, we solve the general and hitherto open problem of constructing a ``strict'' $(n+1)$-type ($(n+1)$-, but not $n$-truncated) for every $n \geq -1$.
It then turns out that the ``subuniverse'' of $\UU_n$ which only contains $n$-types is already such a strict $(n+1)$-type itself.
Our construction also shows immediately that the universe $\UU_n$ is not an $n$-type.
We say that the universes are \emph{non-trivial in a high dimension} or \emph{have a high truncation level}, even though both expressions are slightly inaccurate.
Our constructions can be understood as a usage of higher inductive types in a theory that does not support them, which we also briefly discuss.

Further results along these lines will be presented in the Ph.D. thesis of the first-named author~\cite{nicolai:thesis}. 
In particular, it will be shown that the idea of iteratively ``wrapping'' $\bool$ can be made precise with the help of lemmata that we present in this article. 
However, the necessity to encode truncations impredicatively requires an additional universe level and thus leads to weaker results than those we prove here.

\paragraph{Contents}
In Section~\ref{sec:prelim}, we specify the type theory that we work in and explain some notation.
We stick closely to the standard textbook on HoTT~\cite{HoTTbook}.
A proof that the second universe is not a groupoid can be found in Section~\ref{sec:U1}.
We use Section~\ref{sec:pointed} to develop some simple, but very useful theory on pointed types and the interaction of loop spaces with dependent pair and dependent function types.
Section~\ref{sec:not-truncated-construction} contains our main results: we show that universe $\UU_n$ is not $n$-truncated, and we construct a type that is ``strictly'' of truncation level $n+1$, \ie in particular not of level $n$.
Finally, in Section~\ref{sec:conclusions}, we make some concluding remarks.

\paragraph{On polymorphism}
When we say that $\UU_n$ is not an $n$-type, or that we can construct a \emph{strict} $(n+1)$-type, the number $n$ necessarily is a fixed constant. 
This is because HoTT as specified in~\cite{HoTTbook} does not regard universe levels as a type that one can eliminate into, and an expression such as $\prd{n:\N} \neg \istype n (\UU_n)$ is therefore not a type.
The only thing we can do is proving that for any given $n$, the type $\neg \istype n (\UU_n)$ is inhabited. 
We do this by an external induction on $n$, \ie when we prove $\neg \istype {(n+1)} (\UU_{n+1})$, we assume that we already have a derivation of $\neg \istype n (\UU_n)$ and of corresponding lemmata.
From the point of view of the type theory, occurrences of $n$ are always in canonical form $S(\ldots(S 0)\ldots)$, with the length of this expression depending on the current step in the external induction over derivations.  

\paragraph{Agda formalization} 
This article is supplemented by an electronic appendix which contains formalizations of all our results in the programming language and proof assistant \emph{Agda}~\cite{Norell2007Towards}, making use of the HoTT community's Agda library~\cite{hott_library}.
All proofs have been verified to type check in Agda version 2.4.2.

Deserving mention is a subtle difference between the type theory commonly used to develop HoTT~\cite{HoTTbook}, also used in this article, and the one that Agda implements.
While HoTT universes are \emph{cumulative}, \ie $A:\UU$ and $\UU : \UU'$ imply $A : \UU'$, Agda requires explicit \emph{lifting}.
A sour consequence of this is the following: recall that the univalence axiom implies $\id[\UU_{k+1}] {(\id A B)} {(\eqv A B)}$ for types $A, B : \UU_k$.
Note however that this cannot be stated in Agda, the reason being that $\id A B$ lives in $\UU_{k+1}$, while $\eqv A B$ lives in $\UU_k$.
We can still make this statement by first lifting $\eqv A B$ to the universe $\UU_{k+1}$.
To avoid impacting the readability of our code by manifold instances of lifting, we have striven to represent (pointed) type equality by (pointed) equivalence wherever possible in the formalization.
In a theory with proper cumulativity of universes, both versions work equally well. 

Potentially controversial features used by this formalization include $\eta$-rules for record types (\ie dependent pair types) and an internal type representing levels in the universe hierarchy not reflected in HoTT.
The former is not in any way crucial. 
The latter means that the explanations in the paragraph on polymorphism are not applicable in the theory that Agda implements. 
For example, the expression $\prd{n:\N} \neg \istype n (\UU_n)$ indeed \emph{is} a type in Agda. 

This is very fortunate: it allows us to formalize our result as a single Agda program instead of a countably infinite family of such.
Applying the Agda term for the above Agda type to any canonical natural number $n$ reduces to a derivation of $\neg\istype n (\UU_n)$ not using quantification over universe indices.

Other minor differences are explained in detail in the formalization.

\section{Preliminaries} \label{sec:prelim}

We work in the type theory of the \emph{Homotopy Type Theory} textbook~\cite{HoTTbook}.
This book is our main reference and we assume familiarity with it, especially with respect to notation.
However, we explicitly want to show that a hierarchy of univalent universes alone is sufficient to construct types that are not $n$-types, and we do not assume that the theory has truncations, quotients, or (more generally) higher inductive types. 
This amounts to saying that our theory is the version of intensional MLTT that is formally presented in~\cite[Appendix A.2]{HoTTbook}, together with the univalence axiom as specified in~\cite[Appendix A.3.1]{HoTTbook}.
Let us briefly review some of the details of this theory, together with some notation.

\paragraph{Basic concepts of MLTT}
For \emph{strict} (or \emph{judgmental}) equality, we write $\jdeq$.
Identity types are also called \emph{path spaces}, and they are written $\id[A] x y$ or just $\id x y$ (for $x,y : A$).
Applying the eliminator $J$ is called \emph{path induction}.
An important special case is \emph{transport}: if $P$ is a family of types over $A$ and there are $u : \id[A] x y$ as well as $t : P(x)$, then there is $\trans u t : P(y)$.
A path $p : \id x y$ has an \emph{inverse} $\opp p : \id y x$, and for a second path $q : \id y z$, we have the \emph{composition} $p \ct q : \id x z$.

The theory has a hierarchy $\UU_0 : \UU_1 : \UU_2 \ldots$ of universes.
We want to emphasize that this hierarchy is assumed to be \emph{cumulative}: if $A$ is a type in $\UU_n$, then $A$ is also a type in $\UU_{n+1}$.
If we just write $\UU$, the corresponding statement or derivation is to be understood for \emph{any} universe.

We have the basic finite types $\emptyt$, $\unit$, and $\bool$ with inhabitants $\ttt : \unit$ and $\bfalse, \btrue : \bool$, respectively, together with the negation function $\swap : \bool \to \bool$.
For the natural numbers $\N$, note that we assume addition to be defined by recursion on the second argument.
This allows the presentation to follow the traditional convention of writing $n+1$ instead of $1+n$.
Some of the equalities that we claim to hold judgmentally depend on this assumption.

Further, there are dependent function types ($\Pi$), satisfying the judgmental $\eta$ (or \emph{uniqueness}) property, and dependent pair types ($\Sigma$), together with the special case of non-dependent products ($\times$).
We do not require notation for coproducts.

\paragraph{HoTT-specific definitions}
Assume that $A$ is a type.

Given an integer $n \geq -2$, the type $A$ is called an \emph{$n$-type}, or is said to be \emph{$n$-truncated} or of \emph{h-level $(n+2)$}~\cite[Chapter 7.1]{HoTTbook}, if the type $\istype n (A)$ is inhabited. This type is defined by recursion on $n$,
\begin{alignat*}{2}
&\mathsf{is}\mbox{-}{(-2)}\mbox{-}\mathsf{type}(A)&&\defeq   \sm{a : A} \prd{x : A}   \id a x\\
&\istype{(n+1)}(A) &&\defeq   \prd{x,y : A} \istype{n}(\id{x}{y}).
\end{alignat*}
In the special cases that $n$ is $-2$, $-1$, or $0$, we write $\istype {-2}(A)$ (``A is contractible''), $\istype {-1}(A)$ (``$A$ is propositional'' or ``$A$ is a proposition''), and $\istype {0}(A)$ (``$A$ is a set''), respectively.

For a non-dependent function $f$, we write $\isequiv (f)$ for the proposition which states that $f$ is an \emph{equivalence}, defined in any of the ways given in \cite[Chapter 4]{HoTTbook}.
We write $A \simeq B$ for $\sm {f : A \to B} \isequiv (f)$.
For example, the identity $\idfunc [A]$ is an equivalence, and so is the negation function $\swap: \bool \to \bool$.
It is straightforward to define canonical inhabitants $\iseqown \idfunc$ and $\iseqown \swap$ of $\isequiv(\idfunc [A])$ and $\isequiv (\swap)$.

Given a point $a : A$, the pair $(A,a)$ is called a \emph{pointed type} with \emph{underlying type} $A$ and \emph{basepoint} $a$.
Let us write ${\UUp}$ for the type of pointed types with underlying type living in $\UU$~\cite[Definition~2.1.7]{HoTTbook}, that is,
\begin{equation*}
 \UUp \defeq \sm{A : \UU} A.
\end{equation*}
We call ${\UUp}$ the \emph{universe of pointed types} as this matches the intuition.
Note, however, that it is really just a defined type, rather than a primitive of the theory as the universes $\UU_k$ are.
If $(A,a)$ and $(B,b)$ are pointed types, a \emph{pointed function} consists of a map $f : A \to B$ and a proof of $\id{f(a)}{b}$ showing that the basepoint is preserved.
If additionally $f$ is an equivalence, we speak of a \emph{pointed equivalence}.
We call a pointed type $n$-truncated (or an $n$-type, or say that it has truncation level $n$) if its underlying type has that property.

If $(A,a)$ is a pointed type, its \emph{loop space}~\cite[chapter 2.1]{HoTTbook} is the pointed type
\begin{equation*}
 \Omega(A,a) \defeq ((\idauto[A] a) , \refl a),
\end{equation*}
the elements of which are called \emph{loops}.
As $\Omega$ is thus an endomorphism on $\UUp$, it can be composed with itself.
This gives us the \emph{$n$-fold iterated loop space} 
\begin{alignat*}{2}
 &\Omega^0(A,a) &&\defeq (A,a) \\
 &\Omega^{n+1}(A,a) &&\defeq \Omega^n(\Omega(A,a)).
\end{alignat*}

\paragraph{Univalence}
For types $A$ and $B$, there is a canonical map 
\begin{equation*}
 \idtoeqv : (\id A B) \to (\eqv A B),
\end{equation*}
defined by path induction, where $\refl A : \idauto A$ is mapped to $(\idfunc [A] , \iseqown \idfunc)$.
The \emph{uivalence axiom} says that $\idtoeqv$ is an equivalence.
As is standard, we assume the univalence axiom for every universe $\UU_k$, \ie all universes are assumed \emph{univalent}.

It is a well-known and immediate consequence of the univalence axiom that the smallest universe is not a set.
The standard proof goes as follows.
Suppose $\istype{0}(U_0)$.
By definition of $\istype{0}$, this implies $\istype{-1}(\idauto{\bool})$.
Univalence allows us to replace $\idauto{\bool}$ by $\eqvauto \bool$.
However, there are two distinct automorphisms on $\bool$. 
In formulae:
\begin{align*}
\istype{0}(\UU_0)
& \implies \istype{-1}(\idauto{\bool}) \\
& \implies \istype{-1}(\eqvauto \bool) \\
& \implies \id {(\idfunc[\bool] , \iseqown \idfunc)} {(\swap , \iseqown \swap)} \\
& \implies \id{\idfunc[\bool]}{\swap} \\
& \implies \id{\idfunc[\bool] (\btrue)}{\swap (\btrue)} \\
& \implies \id \btrue {\bfalse} \\
& \implies \bot.
\end{align*}

Intuitively, it may appear that the reason why $\UU_0$ is not a set is that an inhabitant of it, namely $\bool$, is already not a proposition.
However, possibly somewhat surprisingly, this simple idea is rather misleading, and the proof of $\neg \istype 1 (\UU_1)$ already requires significantly more thought.

To prove the general version $\neg\istype n (\UU_n)$ for any chosen $n$, we will develop some theory about pointed types.
An important ingredient will be our \emph{local-global looping principle}, allowing us to freely switch between (higher) loops in the universe and families of loops that are indexed over some type.

\paragraph{Equivalences}
We will make frequent use of the following basic equivalences, all of which are directly stated in~\cite{HoTTbook}. 
\begin{enumerate}[label={(E\arabic*)},ref={E\arabic*}]
\item \label{str-ext}
\emph{Strong function extensionality}: for functions $f, g : \prd {A} B$, there is a canonical map from $\id f g$ to $\prd {a : A} (\id {f(a)} {g(a)})$ (usually called $\happly$), defined by path induction.
This map is an equivalence.
\cite[Chapter 4.9]{HoTTbook}

\item \label{sigma-eq}
\emph{Paths between pairs are pairs of paths}: if $(x_1,y_1)$ and $(x_2,y_2)$ are both of type $\sm X Y$, then $\eqvspace{\id{(x_1,y_1)}{(x_2,y_2)}}{\sm {u : \id{x_1}{x_2}}\id{\trans{u}{y_1}}{y_2}}$.
In the case of a non-dependent product $X \times Y$, the latter type simplifies to $(\id{x_1}{x_2}) \times (\id{y_1}{y_2})$.
A special application of this rule concerns pointed types: univalence tells us that, for pointed types $X$ and $Y$, the type of pointed equivalences between them is equivalent to the type of equalities $\id X Y$.
\cite[Theorem 2.7.2]{HoTTbook}

\item \label{forget-contr}
\emph{Neutral contractible components}: if $Y$ depends on $X$ and $Y(x)$ is contractible for all $x$, then $\eqv{\sm X Y}{X}$.
\cite[Lemma 3.11.9 (i)]{HoTTbook}

\end{enumerate}

\section{The second universe is not a 1-type} \label{sec:U1}

In this section, we want to present the proof that $\UU_1$ is not $1$-truncated.
We will not reuse this result later as it will easily follow from more general constructions.
However, the approach we take for this special case contains some of the key ideas and could therefore be supportive for understanding the later developments.

Let us first try to prove $\neg \istype{1}(\UU_1)$ in a similar way as we have proved $\neg \istype{0}(\UU_0)$ in the previous section:
\begin{align*}
\istype{1}(\UU_1) & \implies \istype{0}(\idauto{\UU_0}) \\
 & \implies \istype{0}(\eqvauto{\UU_0}) \\
 & \implies \istype{-1}(\idauto{(\idfunc[\UU_0] , \iseqown \idfunc )}) \\
 & \implies \text{\ldots ?}
\end{align*}
In the attempt above, in the very first step, we have to choose two inhabitants of $\UU_1$ with sufficiently complicated equality type.
We have chosen $\UU_0$ as we have already seen before that $\UU_0$ is not a set.

The problem is that we seem unable to derive a contradiction from the assumption $\istype 0 (\eqvauto {\UU_0})$.
In fact, an expected meta-theoretic result (related to \emph{parametricity}) is that the identity is the only definable auto-equivalence on $\UU_0$, and that we cannot write down a non-trivial proof that it equals itself.
Because of this, we do not even expect that $\istype {-2} (\eqvauto {\UU_0})$ implies a contradiction (although it does if we assume the law of excluded middle for propositions).
It belongs to a collection of meta-theoretic properties that, to the best of our knowledge, have not been proven rigorously in the presence of univalence so far, but are commonly believed to hold.

This motivates a more well-behaved choice for the problematic first step.
We use the type of loops in $\UU_0$,  
\begin{equation*}
  L \; \defeq \; \sm{X : \UU_0} \idauto{X}.
\end{equation*}
Showing that the second universe is not a groupoid proceeds as follows:
\begin{alignat*}{4}
&&&&& \istype{1}(\UU_1) \\
&\quad & \implies &\quad &&\istype{0}(\idauto{L}) \\
\text{(by univalence)}
 && \implies & &&\istype{0}(\eqvauto{L}) \\
\text{(choose the identity)}
 && \implies & &&\istype{-1}(\idauto{(\idfunc[L], \iseqown \idfunc)}) \\
\intertext{Here, we have a type of paths between pairs.
By~(\ref{sigma-eq}), this corresponds to pairs of paths.
The second component will be trivial: $\iseqown \idfunc$ lives in a propositional type, and its path type will thus be contractible.
We apply (\ref{forget-contr}) and conclude that the type of paths between two equivalences is equivalent to the type of paths between the underlying functions.}
 && \implies &&& \istype{-1}(\idauto{\idfunc[L]}) \\
\text{(by \ref{str-ext})} 
 && \implies &&& \istype{-1}(\prd{a : L} \idauto{a})\\
\text{(unfold $L$ and curry)}
 && \implies &&& \istype{-1}(\prd{X : \UU_0}\prd{p : \idauto{X}} \; \idauto{(X, p)}) \\
\text{(by~\ref{sigma-eq})}
 && \implies &&& \istype{-1}\left(\prd{X : \UU_0}\prd{p : \idauto{X}} \sm{q : \idauto{X}} \; \id{\trans{q}{p}}{p}\right)
\end{alignat*}
It is a standard lemma that transporting a path along a path can be written as path composition~\cite[Theorem~2.11.5]{HoTTbook}: $\id{\trans{q}{p}}{\opp{q} \ct p \ct q}$.
Making this replacement and precomposing with $q$, we get $\istype{-1}(K)$ where
\begin{equation*}
K \; \defeq \; \prd{X : \UU_0}\prd{p : \idauto{X}} \sm{q : \idauto{X}} \; \id{p \ct q}{q \ct p}.
\end{equation*}
Two inhabitants of $K$ are
\begin{align*}
 &\alpha \defeq \lam{X} \lam {p} (\refl{X}, u),\\
 &\beta  \defeq \lam{X} \lam {p} (p, \refl{p \ct p})
\end{align*}
where $u$ is a proof of $\id{p \ct \refl{X}}{\refl{X} \ct p}$.
Since $K$ is propositional, we may conclude $\id[K]{\alpha}{\beta}$.
Choosing $\bool$ for $X$, this implies that any proof of $\idauto \bool$ is equal to $\refl{\bool}$, and we get the same 
contradiction as we got in the proof of $\neg \istype 0 (\UU_0)$.

In the general case, we consider higher loops in higher universes.
The core obstacle in translating the above proof is the step where $\trans{q}{p} = p$ is observed to hold for $q \defeq \refl{X}$ and $q \defeq p$ by virtue of $\trans{q}{p} = \opp{q} \ct p \ct q$.
In general, it is not so clear how a uniform presentation of transporting along higher loops would look like, and it seems unlikely that the analogue of $q$ would admit non-canonical choices even in one dimension higher.
However, as we will see below, this obstacle can effectively be bypassed for higher dimensions.

\section{Pointed Types} \label{sec:pointed}

Pointed types, as defined in~\cite{HoTTbook}, are a simple but helpful concept.
Their properties can usually easily be formulated in terms of ordinary types.
For our presentation, we will develop some of their theory explicitly in this section, aiming to express elegantly how $\Omega$ interacts with $\Sigma$ and $\Pi$.

\subsection{Dependent Pairs and Loops}

We will first treat the interaction of $\Sigma$ and $\Omega$.
Let us begin by recalling the following definition:
\begin{definition}[{pointed family, \cf~\cite[Definition~5.8.1]{HoTTbook}}]
 For a pointed type $\pted{A} \jdeq (A, a)$, a \emph{pointed family} is a type family $P : A \to \UU$ where the type over the basepoint is again pointed:
 \begin{equation*}
   \PointedPred{\pted A} \defeq \sm {P : A \to \UU} P(a).  
 \end{equation*}
 Extending the notion of truncatedness from types to families, we say that the pointed family $(P,p)$ is $n$-truncated if $P$ is a family of $n$-types.
\end{definition}

\begin{remark}
 The definition of a pointed family is identical to that of a \emph{pointed predicate}~\cite[Definition~5.8.1]{HoTTbook}.
However, we want the reader to think of actual families, and \emph{predicates} are usually understood as ``logical'' (propositional) properties.
 Note that a pointed type can always be seen as a pointed family over the trivial pointed type $(\unit , \ttt)$.
\end{remark}

Let $(P, p)$ be a pointed family over some pointed type $(A,a)$.
There is an induced type family $\tilde P$ over $\Omega (A,a)$ given by $\tilde P(q) \defeq \id[P(a)]{\trans q p}{p}$.
The type over the basepoint is $P(\refl a) \jdeq (\idauto p)$ and therefore trivially inhabited by reflexivity.
This allows us to define
a fibred version of $\Omega$:

\begin{definition}[$\tilde\Omega$]
 For a pointed type $\pted A \jdeq (A,a)$, we define
 \begin{align*}
  &\tilde \Omega : \PointedPred {\pted A} \to \PointedPred {\Omega {\pted A}} \\
  &\tilde\Omega (P,p) \defeq \pair   {\lam q \id[P(a)]{\trans q p}{p}}   {\refl {p}} .
 \end{align*}
 Consequently, $\Omega$ and $\tilde \Omega$ together form the following endofunction:
 \begin{align*}
  &\langle \Omega , \tilde\Omega \rangle : \sm {{\pted A} : \UUp} \PointedPred {\pted A} \to \sm {{\pted A} : \UUp} \PointedPred {\pted A} \\
  &\langle \Omega , \tilde\Omega \rangle (\pted A , \pted P) \jdeq (\Omega {\pted A} , \tilde \Omega {\pted P})
 \end{align*}
\end{definition}
Given a pair of a pointed type and a pointed family, it is straightforward to construct a pointed type corresponding to the dependent sum.
\begin{definition}[$\sigbu$]
 We define the operator $\sigbu$ in the following way:
 \begin{align*}
  &\sigbu : (\sm {{\pted A} : \UUp} \PointedPred {\pted A}) \to \UUp \\
  &\sigbu ((A,a) , (P,p)) \defeq    \pair {\sm {A} P}  {(a , p)}
 \end{align*}
 We write $\sigbu_{\pted A} \pted P$ synonymously for $\sigbu (\pted A, \pted P)$.
\end{definition}

We are now ready to formulate precisely how dependent sums and loop spaces interact.
\begin{lemma} \label{om-si-comm}
 The operators $\sigbu$ and $\Omega$ commute in the following sense:
 \begin{equation*}
  \id{\Omega \functioncompose \sigbu}{\sigbu \functioncompose \langle \Omega , \tilde \Omega \rangle}.
 \end{equation*}
\end{lemma}
\begin{proof}
 Let $\pted A \jdeq (A,a)$ be a pointed type with a pointed family $\pted P \jdeq (P,p)$.
By function extensionality (\ref{str-ext}), it is enough to show that both sides of the equation are equal if applied to $(\pted A, \pted P)$.
Let us calculate: 
\begin{alignat*}{4}
 && &&& (\Omega \functioncompose \sigbu)(\pted A , \pted P)\\
  \text{(by definition of $\sigbu$)} && \quad & \jdeq & \quad & \Omega \pair{\sm A P} {(a,p)} \\
  \text{(by definition of $\Omega$)} &&& \jdeq && \pair {\idauto {(a,p)}} {\refl {(a,p)}} \\
  \text{(by~\ref{sigma-eq})} &&& \; \idsym && \pair {\sm {q : \idauto a} \id{\trans{q}{p}}{p}} {(\refl a , \refl p)} \\
  \text{(by definition of $\sigbu$)} &&& \jdeq && \sigbu_{(\idauto a , \refl a)} (\lam q \id[P(a)]{\trans q p}{p} , \refl p) \\
  \text{(by definition of $\Omega$ and $\tilde\Omega$)} &&& \jdeq && (\sigbu \functioncompose \langle \Omega , \tilde\Omega \rangle) (\pted A , \pted P).  
\end{alignat*}
\end{proof}

\subsection{Dependent Functions and Loops}

The situation is similar, and even simpler, if we want to examine the interaction of $\Pi$ and $\Omega$.
Given a family of pointed types over $A$, there is a straightforward way to construct a pointed type out of the given data corresponding to the dependent function type.

\begin{definition}[$\pibu$]
 We define the operator $\pibu$ by:
\begin{align*}
 &\pibu : (\sm {A : \UU} (A \to \UUp)) \to \UUp \\
 &\pibu (A, \pted F) \defeq \pair {\prd {A} \fst \functioncompose \pted F}  {\snd \functioncompose \pted F}
\end{align*}
We use the notations $\pibucurried {a:A}{\pted{F}(a)}$ and $\pibu_A \pted F$ synonymously for $\pibu (A, \pted F)$.
Note that the type $A$ is \emph{not} pointed.
\end{definition}

With this at hand, we are ready to prove:

\begin{lemma} \label{om-pi-comm}
 $\Omega$ and $\pibu$ commute in the following sense: given a type $A$ and a family $\pted F$ of pointed types over $A$, we have
 \begin{equation*}
  \id{\Omega (\pibu (A , \pted F))}{\pibu (A , \Omega \functioncompose \pted F)}.
 \end{equation*}
\end{lemma}
\begin{proof}
Let us do the following calculation:
\begin{alignat*}{4}
 && &&& \Omega (\pibu (A , \pted F)) \\
  \text{(by definition of $\pibu$)} && \quad & \jdeq & \quad & \Omega   \pair  {\prd {A} \fst \functioncompose \pted F}  {\snd \functioncompose \pted F}  \\
  \text{(by definition of $\Omega$)} &&& \jdeq && \bigpair {\idauto {\snd \functioncompose \pted F}}  {\refl {\snd \circ \pted F}}  \\
  \text{(by~\ref{str-ext})} &&& \; \idsym && \bigpair{\prd {a:A} \idauto {\snd (\pted F(a))}}     {\lam {a} \refl {\snd (\pted F(a))}  }  \\ 
  \text{(by definition of $\pibu$)} &&& \jdeq &&  \pibu (A , \Omega \functioncompose \pted F).    
\end{alignat*}
\end{proof}

\section{Homotopically Complicated Types} \label{sec:not-truncated-construction}

In this section, we prove the main results of this article. 
We construct a type that \emph{strictly} has truncation level $(n+1)$, namely the type of $(n+1)$-loops in $\UU_n$ which are based at $n$-types. 
From that, it will easily follows that $\UUtt n n$ is a strict $(n+1)$-type as well, and that $\UU_n$ is not $n$-truncated. 

We begin with a lemma that tells us how a truncated $\Sigma$-component can be neutralized by $\Omega$.
\begin{lemma}\label{forget-lemma}
 Let $n$ be a natural number.
 Further, let $\pted A$ be a pointed type and $\pted P$ be a pointed family over $\pted A$ of truncation level $n-2$.
Then,
 \begin{equation*}
  \id{\Omega^n (\sigbu_{\pted A} \pted P)}{\Omega^n(\pted A)}.
 \end{equation*}
\end{lemma}
\begin{proof}
 We do induction on $n$.
For the base case $n \jdeq 0$, the statement is exactly given by~(\ref{forget-contr}).
For the induction case, we have the following chain of equalities:
\begin{alignat*}{4}
 &&& && \Omega^{n+1}(\sigbu_{\pted A} \pted P) \\
  \text{} &&& \jdeq &&  \Omega^{n}(\Omega(\sigbu_{\pted A} \pted P)) \\
  \text{(by Lemma~\ref{om-si-comm})} && \quad & \; \idsym & \quad &   \Omega^{n} (\sigbu_{\Omega \pted A}  \tilde\Omega \pted P)    \\
  \text{(by induction hypothesis)} &&& \; \idsym &&  \Omega^{n}(\Omega(\pted A)) \\
  \text{} &&& \jdeq &&  \Omega^{n+1}(\pted A)
\end{alignat*}
For the penultimate step, note that if $\pted P$ is $(n-1)$-truncated, then $\tilde\Omega \pted P$ is $(n-2)$-truncated.
\end{proof}

We are now ready to prove our local-global looping principle, stating that a loop in the universe is the same as a family of loops in the basepoint:
\begin{lemma}[local-global looping] \label{local-global}
 Let $A$ be a type in a universe $\UU$ and $n$ be a natural number.
Then,
 \begin{equation*}
   \id{\Omega^{n+2}(\UU , A)}{\pibucurried {a:A} {\Omega^{n+1}(A,a)}}.
 \end{equation*}
\end{lemma}
\begin{proof}
 The proof is again done by a calculation, utilizing most of the theory we have developed so far:
\begin{alignat*}{4}
 &&& && \Omega^{n+2}(\UU , A) \\
  \text{(by definition)} && \quad & \jdeq & \quad &  \Omega^{n+1} \pair{\idauto A}  {\refl A}  \\
  \text{(by univalence)} &&& \; \idsym && \Omega^{n+1} \pair{\eqvauto A}{(\idfunc [A] , \iseqown \idfunc)}                    \\
  \text{(by definition of $\eqvsym$)} &&& \jdeq && \Omega^{n+1} \pair{\sm {f : A \to A} \isequiv (f)}{(\idfunc [A] , \iseqown \idfunc)}  \\
  \text{(by definition of $\sigbu$)} &&& \jdeq && \Omega^{n+1}   (  \sigbu_{(A \to A , \idfunc [A])} (\isequiv  ,  \iseqown \idfunc)   )    \\ 
  \text{(by Lemma~\ref{forget-lemma})} &&& \; \idsym && \Omega^{n+1}(A \to A , \idfunc [A]) &&  \\
  \text{(by definition of $\pibu$)} &&& \jdeq && \Omega^{n+1}(\pibucurried {a:A}{(A,a)}) \\
  \text{(by Lemma~\ref{om-pi-comm})} &&& \; \idsym && \pibucurried{a:A}{\Omega^{n+1}(A , a)}.  
\end{alignat*}
\end{proof}
Note that the ``$+2$'' (respectively ``$+1$'') in the statement of Lemma~\ref{local-global} is necessary. This is because $\Omega(\UU,A)$ and $\pibucurried {a:A} {(A,a)}$ are in general not the same, as the latter can be simplified to $(A \to A , \idfunc[A])$.
In the given proof, the step where Lemma~\ref{forget-lemma} is applied would fail.

It will be useful to consider the restriction of a universe to its $n$-types:
\begin{definition}[{$\UUt n$, \cf~\cite[Chapter 7.1]{HoTTbook}}]
 For a universe $\UU$ and an integer $n \geq -2$, we define $\UUt n$ as the ``subuniverse'' of $n$-types, that is,
 \begin{equation*}
  \UUt n  \defeq \sm {A : \UU} \istype n (A).
 \end{equation*}
 Similar as for $\UUp$, we call $\UUt n$ a \emph{universe}, but it is important to note that it is a defined type, not a primitive of the theory.
\end{definition}
The following two simple and well-known observations will be useful:
\begin{lemma}[{\kern-0.4em\cite[Theorem~7.2.9]{HoTTbook}}] \label{lem:n-type-if-loop-contr}
 For $n \geq -1$, a type $A$ is an $n$-type if and only if for every $a$ in $A$, the loop space $\Omega^{n+1}(A,a)$ is contractible.  
\end{lemma}
\begin{lemma}[{\kern-0.4em\cite[Theorem~7.1.11]{HoTTbook}}] \label{Un-is-Sn-type}
 For any $n \geq -2$ and universe $\UU$, the type $\UUt n$ is $(n+1)$-truncated.  
\end{lemma}
For $n \in \N$, let us write $P_n(X)$ for the type of $(n+1)$-loops that live in the universe $\UUtt n n$ and have basepoint $X$.
More precisely, we abbreviate
\begin{align*}
 & P_n : {\UUtt n n} \to \UUp_{n+1} \\
 & P_n(X) \defeq \Omega^{n+1}(\UUtt n n , X).
\end{align*}
Homotopically, these loops $P_n(X)$ are rather tame:
\begin{corollary}[of Lemma~\ref{Un-is-Sn-type}] \label{Pn-are-sets}
 $P_n$ is a family of sets, that is,
 \begin{equation*}
  \prd {\UUtt n n} \istype 0 \functioncompose P_n.
 \end{equation*}
\end{corollary}

Still for $n \in \N$, an $(n+1)$-loop consists of a basepoint $X$ and the actual loop around $X$.
The type of \emph{$(n+1)$-loops in universe $\UUtt n n$} is therefore given by 
\begin{align*}
 \loops_n &: \UU_{n+1} \\
 \loops_n &\defeq \sm {\UUtt n n} \fst \functioncompose P_n.
\end{align*}
We additionally define $\loops_{-1} \defeq \bool$ in the lowest universe $\UU_0$, which makes it possible to treat all universes uniformly instead of only those above $\UU_0$.

This type is also fairly tame homotopically:
\begin{lemma} \label{L-is-trunc}
 For all natural numbers $n$, the type $\loops_{n-1}$ is $n$-truncated, that is, we can construct
 \begin{equation*}
  h_n : \; \istype {n} (\loops_{n-1}).
 \end{equation*}
\end{lemma}
\begin{proof}
 The claim is clearly fulfilled for $n \jdeq 0$, so let us assume $n \geq 1$.
 By a standard lemma~\cite[Theorem~7.1.8]{HoTTbook}, it is enough to examine the two parts of the dependent pair type separately.
The required property for the first part is given by Lemma~\ref{Un-is-Sn-type}.
The second component is a family of sets by Corollary~\ref{Pn-are-sets}, which suffices as $n \geq 0$~\cite[Theorem~7.1.7]{HoTTbook}.
\end{proof}

For a pointed type $(A,a)$, we say that an element $b : A$ is \emph{trivial} if it is equal to the basepoint, \ie if we have a proof of $\id a b$.
\begin{lemma} \label{non-trivial-exists}
 For all $n \geq 0$, the type $\Omega^{n+1} \pair{\UU_{n}} {\loops_{n-1}}$ has a non-trivial inhabitant.
The same is true for $\Omega^{n+1} \pair{\UUtt {n}{n}}{(\loops_{n-1} , h_{n})}$.
\end{lemma}

\begin{proof}
 Observe that the pointed type $\pair{\UUtt {n}{n}}{(\loops_{n-1} , h_{n})}$ can be written as (and is judgmentally equal to) the expression $\sigbu_{(\UU_{n} , \loops_{n-1})} (\istype n , h_n)$.
As the predicate $\istype n$ is propositional, Lemma~\ref{forget-lemma} implies the equivalence of the two loop spaces of this lemma, and we may restrict ourselves to showing that the claim holds for $\Omega^{n+1} \pair{\UU_{n}} {\loops_{n-1}}$.

 We do induction on $n$.
 For $n \jdeq 0$, we have to provide a non-trivial inhabitant of $\idauto 2$.
This is $\swap$, which is different from the trivial inhabitant $\refl \bool$ (see Section~\ref{sec:prelim}).

 Assume $n \jdeq m + 1$ and calculate:
 \begin{alignat*}{4}
 &&& && \Omega^{m+2}(\UU_{m+1} , \loops_m) \\
  \text{(by~\ref{local-global} - local-global)} && \quad & \; \idsym  & \quad & \pibucurried {(X,q) : \loops_m} {\Omega^{m+1}(\loops_m , (X,q))}  \\
  \text{(by definition of $\sigbu$)} &&& \jdeq && \pibucurried {(X,q) : \loops_m} {\Omega^{m+1} ({\sigbu_{(\UUtt m m , X)} (\fst \functioncompose P_m , q)}) }\\
  \text{(by~\ref{om-si-comm} - $\Omega$, $\sigbu$ commute)} &&& \; \idsym && \pibucurried {(X,q) : \loops_m} 
      {\sigbu_{\Omega^{m+1} (\UUtt m m , X)} {\tilde\Omega^{m+1} (\fst \functioncompose P_m , q)}}
\end{alignat*}
The underlying type of this last pointed type has the following inhabitant:
\begin{equation*}
 \xi \defeq \lam {(X,q)} (q, d_q)
\end{equation*}
where $d_q$ is defined as follows:
\begin{itemize}
\item
For $m \jdeq 0$, the type of $d_q$ is $\id {\trans q q}{q}$, which is easily seen to be inhabited after writing $\trans q q$ as $\opp q \ct q \ct q$.
Note that this case corresponds to the special case we discussed in Section~\ref{sec:U1}.

\item
For $m \geq 1$, the type of $d_q$ is contractible by Corollary~\ref{Pn-are-sets} and the definition of $\tilde \Omega$, providing a canonical choice for $d_q$.
\end{itemize}
We claim that $\xi$ is non-trivial.
Once again, let us view $\loops_{m-1}$, together with $h_m$, as an inhabitant of $\UUtt m m$.
Recall that we have defined $P_m(\loops_{m-1} , h_m)$ to be $\Omega^{m+1} \pair{\UUtt m m} {(\loops_{m-1} , h_m)}$.
By the induction hypothesis, we can construct a non-trivial inhabitant $\tilde q$ of the underlying type, so that we have
\begin{equation} \label{loops-inhabitant}
\pair{(\loops_{m-1},h_m)}{\tilde q} : \loops_m.
\end{equation}
If $\xi$ was trivial, the term $\fst (\xi (X,q)) \jdeq q$ would be trivial in $P_m(X)$ for any $(X,q) : \loops_m$.
But this is invalidated by (\ref{loops-inhabitant}).
\end{proof}

This allows us to prove:
\begin{theorem}\label{theorem-one}
 In Martin-L\"{o}f type theory with a hierarchy of univalent universes $\UU_0, \UU_1, \UU_2, \ldots$, the universe $\UU_n$ is not an $n$-type.
That is, for any natural number $n$, the type
 \begin{equation*}
  \neg \istype n \left( \UU_n \right)
 \end{equation*}
is inhabited.
\end{theorem}
\begin{proof}
 If $\UU_n$ was an $n$-type, then $\Omega^{n+1} \pair{\UU_{n}} {\loops_{n-1}}$ would be propositional, contradicting Lemma~\ref{non-trivial-exists}.
\end{proof}

At the same time, we have solved the question of constructing a ``strict'' $n$-type:
\begin{theorem} \label{theorem-two}
 For a given $n \geq -2$, there is (in the settings of Theorem~\ref{theorem-one}) a type that is an $(n+1)$-type but not an $n$-type.
In particular, for $n \geq -1$, the type $\loops_n$ has this property.
 Further, for $n \geq 0$, the universe of $n$-types at level $n$, namely $\UUtt n n$, is such a strict $(n+1)$-type.
\end{theorem}
\begin{proof}
 For $n \jdeq -2$, the empty type proves the statement.
The claim for $\UUtt n n$ follows in the same way as Theorem~\ref{theorem-one}, combined with Lemma~\ref{Un-is-Sn-type}.
$\loops_{-1} \jdeq \bool$ is clearly strictly a set.
For $n \geq 0$, Lemma~\ref{L-is-trunc} shows that $\loops_n$ is $(n+1)$-truncated.
To see that it is not $n$-truncated, observe that the first component is $\UUtt n n$ and therefore not $n$-truncated, while the second component is always inhabited.
\end{proof}

\section{Concluding Remarks} \label{sec:conclusions}

We believe that Theorems~\ref{theorem-one} and~\ref{theorem-two} are as strong as they can be, in the sense that $\UU_n$ can in neither case be replaced by a smaller universe.
More concretely, it should be consistent to assume that every type in $\UU_n$ is $n$-truncated (for any given $n$).
We are not aware of any published proof of this fact, but one approach would be to use that the hierarchy $\UUtt 0 0, \UUtt 1 1, \ldots$ is (in an appropriate sense) closed under all type formers, including universe formation by Lemma~\ref{Un-is-Sn-type}.
This makes it easy to construct a model in which the claimed property holds.

This of course is not true if we consider a theory with so-called \emph{higher inductive types} (HITs)~\cite[Chapter 6]{HoTTbook}.
While univalence is certainly the single most characteristic and powerful concept that homotopy type theory introduces, 
HITs are another interesting feature that is often discussed.
They allow inductive types to be equipped with constructors of (higher) paths rather than only points, and it is not surprising that they can be used to construct types that are not $n$-truncated for a given $n$.
A canonical candidate for a HIT that is not an $(n-1)$-type is the sphere $\Sn^{n}$ which can be generated with one point-constructor $\base$ and one path-constructor $\lloop$ that gives an inhabitant of the underlying type of $\Omega^{n}(\Sn^{n}, \base)$.
Unfortunately, even the seemingly simple statement that $\lloop$ is non-trivial is not amendable to an instant argument.
While $\Sn^n$ has $\mathbb Z$ as $n$-th homotopy group which immediately implies that it is not an $(n-1)$-type, 
calculating it it in HoTT requires some effort, for example using the long exact sequence~\cite{blog_overview}.

On the other hand, the construction of $\loops_n$ that we present in this article can be understood as a way to use spheres even if the theory does not support HITs.
Recall that our type $\loops_n$ was (after unfolding the definition of $P_n$) defined as
\begin{equation*}
 \loops_n \defeq \sm{X : \UUtt n n} \fst(\Omega^{n+1}(\UUtt n n , X)).
\end{equation*}
If HITs are available, $\loops_n$ is equivalent to the function type
\begin{equation*}
 \Sn^n \to \UUtt n n.
\end{equation*}
Even if we do not have $\Sn^n$ available in the theory, we can thus still talk about how the sphere could be mapped into another type (this is the case for any non-recursive HIT).

As mentioned above, further related results will be available in~\cite{nicolai:thesis}, in particular a proof that the type $\boolpower n$ is a strict $n$-type, where
 \begin{align*}
  &\boolpower 0 \defeq \bool \\
  &\boolpower{n+1} \defeq \sm{X : \UU_n} \brck{\id{X} {\boolpower n}}.
 \end{align*}
A nice aspect (compared to Lemma~\ref{non-trivial-exists}) of this solution is that the loop space $\Omega^n(\boolpower n, x)$ is equivalent to $\bool$ for any basepoint $x$. 
The truncation can be encoded impredicatively, although this is trickier than one expects.
Unfortunately, it causes the results to be weaker by one universe level than those of the Theorems~\ref{theorem-one} and~\ref{theorem-two}.

\vspace{1em}
\paragraph*{\textbf{Acknowledgments}}
First of all, we would like to thank the participants, particularly the organizers, of the Univalent Foundations Program in Princeton 2012/2013 for letting us know of the problem of constructing a ``strict'' $n$-type without higher inductive types, and for many beneficial discussions.
We especially thank Thierry Coquand and Thorsten Altenkirch for their stimulating comments and for encouraging us to continue pursuing these results.
The anonymous reviewers of \emph{ACM Transactions on Computational Logic} have given very helpful remarks, and we thank them for all the suggestions that enabled us to improve the presentation of this article.

The second-named author would like to acknowledge the following: 
This material is based on research sponsored by the Air Force Research Laboratory, under agreement number FA8655-13-1-3038.
The U.S.\ Government is authorized to reproduce and distribute reprints for Governmental purposes notwithstanding any copyright notation thereon.
The views and conclusions contained herein are those of the authors and should not be interpreted as necessarily representing the official  policies or endorsements, either expressed or implied, of the Air Force  Research Laboratory or the U.S. Government.

\bibliographystyle{plain}
\bibliography{master}

\end{document}